\documentclass[11pt%,draft
]{article}

\usepackage[all]{xy}

\usepackage[latin1]{inputenc}
\usepackage{amsfonts}
\usepackage{amsmath}
\usepackage{amssymb}
\usepackage{amscd}
\usepackage{amsthm}
\usepackage{indentfirst}
\usepackage[hmargin=2.98cm
,vmargin=3.0cm
]{geometry}

\usepackage{epsfig}

\newtheorem{thm}{Theorem}[section]
\newtheorem{lemma}[thm]{Lemma}
\newtheorem{prop}[thm]{Proposition}
\newtheorem{coroll}[thm]{Corollary}

\theoremstyle{definition}

\newtheorem{defin}[thm]{Definition}
\newtheorem{rem}[thm]{Remark}
\newtheorem{exam}[thm]{Example}

\newtheorem{question}[thm]{Question}

\newtheorem*{acknow}{Acknowledgements}

\newcommand{\R}{{\mathbb{R}}}

\newcommand{\Z}{{\mathbb{Z}}}
\newcommand{\N}{{\mathbb{N}}}
\newcommand{\C}{{\mathbb{C}}}

\newcommand{\cA}{{\mathcal{A}}}
\newcommand{\cB}{{\mathcal{B}}}
\newcommand{\cC}{{\mathcal{C}}}

\newcommand{\cL}{{\mathcal{L}}}

\newcommand{\cO}{{\mathcal{O}}}
\newcommand{\cP}{{\mathcal{P}}}
\newcommand{\cQ}{{\mathcal{Q}}}

\newcommand{\cT}{{\mathcal{T}}}

\newcommand{\fc}{{:\ }}

\newcommand{\ol}{\overline}
\newcommand{\wt}{\widetilde}
\newcommand{\wh}{\widehat}

\newcommand{\tb}{\textbf}

\DeclareMathOperator{\id}{id}

\DeclareMathOperator{\supp}{supp}

\DeclareMathOperator{\Ham}{Ham}

\DeclareMathOperator{\Lip}{Lip}

\DeclareMathOperator{\diam}{diam}

\DeclareMathOperator{\Spec}{Spec}

\title{Approximation of quasi-states on manifolds}
\author{Adi Dickstein\footnote{School of Mathematical Sciences, Tel Aviv University.}\ \ and Frol Zapolsky\footnote{Department of Mathematics, University of Haifa.}}
\date{}

\begin{document}

\renewcommand{\labelenumi}{(\roman{enumi})}

\maketitle

\begin{abstract}
Quasi-states are certain not necessarily linear functionals on the space of continuous functions on a compact Hausdorff space. They were discovered as a part of an attempt to understand the axioms of quantum mechanics due to von Neumann. A very interesting and fundamental example is given by the so-called median quasi-state on $S^2$. In this paper we present an algorithm which numerically computes it to any specified accuracy. The error estimate of the algorithm crucially relies on metric continuity properties of a map, which constructs quasi-states from probability measures, with respect to appropriate Wasserstein metrics. We close with non-approximation results, particularly for symplectic quasi-states.
\end{abstract}

\section{Introduction and results}\label{s:Intro_results}

Quasi-states on a topological space are a certain generalization of integration against a Borel probability measure. These objects arose from an attempt to understand von Neumann's axioms of quantum mechanics, according to which a quantum system is found in a \emph{state}, which is a positive \emph{linear} functional on the algebra of observables, which are bounded linear operators on a complex Hilbert space. However, physicists objected to this linearity, because it is meaningless from the point of view of physics, unless the two observables in question commute. Thus the notion of a quantum quasi-state appeared, wherein the linearity assumption was replaced by the less restrictive assumption of linearity on commuting observables. The natural question whether nonlinear quantum quasi-states exist was finally settled by Gleason \cite{Gleason_Meas_closed_subsp_Hilb_sp}, who showed that quantum quasi-states are linear provided the underlyng Hilbert space has dimension at least three. See also \cite{Entov_Polterovich_Zapolsky_An_anti_Gleason_phenomenon_simul_measurements_classical_mech} and references therein.

The quantum-classical correspondence principle says that a classical mechanical system is in a certain sense the limit of quantum systems when the Planck constant $\hbar$, considered as a parameter, tends to $0$. Since there are no non-linear quantum quasi-states, a natural question therefore is whether \emph{classical} quasi-states exist. This was answered in the positive by Aarnes in \cite{Aarnes_Qss_qms}. Later, Entov and Polterovich discovered an even more stringent subtype of \emph{symplectic} quasi-states \cite{Entov_Polterovich_Quasi_states_symplectic_intersections}. Remarkably, when the underlying symplectic manifold is the $2$-sphere, what they obtained is the so-called \emph{median quasi-state}, which also appears indirectly in Aarnes's topological theory. It is therefore among the most fundamental examples of quasi-states, and also among the simplest. It is described in Example \ref{exam:median_Aarnes_qss}.

Due to the fundamental nature and interest of the median quasi-state, it is desirable to be able to numerically compute it. Our main contribution in this paper is an algorithm with does that to any specified accuracy. We estimate the error of the algorithm using metric continuity of a natural construction of quasi-states from measures. Other results here are that, \emph{inter alia}, certain quasi-states on symplectic manifolds of dimension $\geq 4$, coming from Floer theory \cite{Entov_Polterovich_Symp_QSs_semisimplicity_QH}, cannot be approximated by this construction. Curiously, Wasserstein metrics make an appearance here.

Quasi-states had long remained a largely theoretical topic, but this is starting to change. For instance, in \cite{Rustad_The_median_multidim_spaces} this concept is applied to the definition of a multidimensional median. See also the references therein. Our paper also brings this topic closer to the applied side.

\subsection{Basic examples and the approximation algorithm}\label{ss:basic_exs_approx_algorithm}

Let $X$ be a compact Hausdorff space, and let $C(X)$ be the Banach algebra of real-valued continuous functions on $X$, endowed with the $C^0$-norm $\|f\|_{C^0} = \max_{x \in X}|f(x)|$. For $f \in C(X)$ the subalgebra it generates is defined as $C(f) = \{\phi\circ f\,|\, \phi \fc \R \to \R \text{ continuous}\}$.
\begin{defin}
A quasi-state on $X$ is a functional $\zeta \fc C(X) \to \R$ such that
\begin{enumerate}
  \item $\zeta(1) = 1$;
  \item $\zeta(f) \leq \zeta(g)$ for $f,g \in C(X)$ with $f \leq g$;
  \item for each $f\in C(X)$, $\zeta$ is linear on $C(f)$.
\end{enumerate}
We moreover call $\zeta$ simple if $\zeta(f^2) = \zeta(f)^2$ for all $f \in C(X)$.
\end{defin}

\begin{rem}\label{rem:elementary_pties_qss}
\begin{enumerate}
  \item Note that if $\mu$ is a Borel probability measure on $X$, the functional $\int_X\cdot\,d\mu \fc C(X) \to \R$ defines a linear quasi-state, which is simple if and only if $\mu$ is a delta-measure.
  \item It easily follows from the definition that a quasi-state $\zeta$ on $X$ is $1$-Lipschitz with respect to the $C^0$-norm, that is $|\zeta(f) - \zeta(g)| \leq \|f-g\|_{C^0}$ for $f,g \in C(X)$.
\end{enumerate}
\end{rem}

The existence of nonlinear quasi-states is due to Aarnes \cite{Aarnes_Qss_qms}. He developed a general construction of quasi-states, which we will describe in detail in Section \ref{ss:Construction_qss_tms} below.

Let us start with two interesting examples of simple quasi-states on the sphere $S^2 \subset \R^3$. A quasi-state is Lipschitz with respect to the $C^0$-norm by Remark \ref{rem:elementary_pties_qss}, therefore uniquely defined by its values on the $C^0$-dense subset of Morse functions. Let us fix a Morse function $f \fc S^2 \to \R$.
\begin{exam}\label{exam:median_Aarnes_qss}
\begin{enumerate}
  \item \tb{The median quasi-state $\zeta$:} Let $\nu$ be the normalized Lebesgue measure on $S^2$ coming from the standard round area form. There is a unique component $m_\nu(f)$ of a level set of $f$, called the median of $f$, with the property that $\nu(C) \leq \frac 1 2$ for every connected component $C$ of $S^2 - m_\nu(f)$. The value of the median quasi-state $\zeta$ at $f$ is $\zeta(f) = f(m_\nu(f))$.
  \item \tb{The Aarnes quasi-state $\zeta_Z$} corresponding to a finite subset $Z \subset S^2$ of odd cardinality $|Z| \geq 3$: let $\mu_Z = \frac{1}{|Z|}\sum_{z\in Z}\delta_z$ be the uniformly distributed discrete probability measure supported on $Z$. Then there is a unique component $m_Z(f)$ of a level set of $f$, called the $Z$-median of $f$, such that every connected component $C$ of $S^2 - m_Z(f)$ has $\mu_Z(C) \leq \frac 1 2$. The value of the Aarnes quasi-state $\zeta_Z$ at $f$ is $\zeta_Z(f) = f(m_Z(f))$.
\end{enumerate}
\end{exam}
\noindent As we mentioned above, the median quasi-state arises independently from quasi-morphisms on the group of area- and orientation-preserving diffeomorphisms of $(S^2,\nu)$, which are construction using Floer theory \cite{Entov_Polterovich_Calabi_quasimorphism_quantum_homology}. It can be characterized as the unique quasi-state on $S^2$ which is invariant under the group of area-preserving diffeomorphisms and which vanishes on functions supported in disks of area at most $\frac 1 2$. As such, it is of fundamental interest, and thus it is desirable to be able to numerically compute it on a wide class of functions.

Let us endow $S^2 \subset \R^3$ with the metric induced from the Euclidean metric on $\R^3$, and for $f \in S^2$ let $\|f\|_{\Lip}$ be its Lipschitz constant with respect to this metric. We can now formulate our main result.
\begin{thm}\label{thm:main_result}
There is an algorithm which accepts as input a Lipschitz function $f \in S^2$ and an integer parameter $N \geq 46$, and produces as output a number which differs from $\zeta(f)$ by no more than
\begin{equation}\label{eq:estimate_main_algorithm}
\frac{C}{N}\cdot\|f\|_{\Lip}\,,\quad \text{where }C\approx 197.778\dots
\end{equation}
The algorithm has complexity $O(N^2\log N)$.
\end{thm}

\begin{rem}
\begin{enumerate}
  \item Alternatively, one can first specify the accuracy $\epsilon > 0$, and the theorem says that there is an algorithm of complexity $O(-\epsilon^{-2}\log\epsilon)$ which produces a number differing from $\zeta(f)$ by at most $\epsilon$.
  \item The constant appearing in \eqref{eq:estimate_main_algorithm} is probably far away from being sharp. Due to the large value of the constant, the estimate is of limited importance in practice. The main point is to prove the existence of a convergent algorithm computing $\zeta(f)$.
  \item In practice, Lipschitz functions on $S^2$ are given for instance by restrictions of polynomials on $\R^3$. If
    $$f = \sum_{i,j,k}c_{ijk}x^iy^jz^k\Big|_{S^2}$$
    is such a function, where $x,y,z$ are the coordinates on $\R^3$, then
    $$\|f\|_{\Lip} \leq D\sqrt{3}\bigg(\sum_{i,j,k}c_{ijk}^2\bigg)^{1/2}\,,$$
    where $D$ is the degree of the initial polynomial.
\end{enumerate}
\end{rem}

The algorithm replaces $f$ with a function $F$ which is piecewise linear with respect to a suitable triangulation of $S^2$, and then algorithmically computes $\zeta_Z(F)$ where $Z$ is a carefully chosen subset of the set of vertices of the triangulation. The estimate \eqref{eq:estimate_main_algorithm} relies on a metric approximation result for quasi-states on manifolds, Theorem \ref{thm:metric_conv_qss} below. Theorem \ref{thm:main_result} is proved in Section \ref{ss:pf_main_thm}.

\subsection{Quasi-states from measures}\label{ss:Construction_qss_tms}

In this section we describe a map $\Psi$ which assigns a simple quasi-state to a measure on a CW complex with vanishing first cohomology, and formulate its continuity properties, which are crucial for the proof of the estimate \eqref{eq:estimate_main_algorithm} of Theorem \ref{thm:main_result}. The median quasi-state and the Aarnes quasi-states of Example \ref{exam:median_Aarnes_qss} are the value of $\Psi$ on the Lebesgue measure on $S^2$ and on the discrete measure $\mu_Z$, respectively.

The map $\Psi$ is constructed in two steps. The first step is the Aarnes representation theorem, proved in \cite{Aarnes_Qss_qms}, which furnishes a bijection between quasi-states and certain set functions, known as topological measures. The second step constructs topological measures from measures.

\subsubsection{Topological measures}\label{sss:top_measures}

Let $\cC(X)$ and $\cO(X)$ be the collections of closed and open subsets of $X$, respectively, and let $\cA(X) = \cC(X) \cup \cO(X)$.

\begin{defin}
A function $\tau \fc \cA(X) \to [0,1]$ is called a topological measure if
\begin{enumerate}
  \item $\tau(X) = 1$;
  \item $\tau(A) \leq \tau(A')$ for $A,A' \in \cA(X)$ with $A \subset A'$;
  \item $\tau(A) = \sum_{i=1}^{m}\tau(A_i)$ for $A,A_1,\dots,A_m \in \cA(X)$, such that the $A_i$ are all mutually disjoint and $A = \bigcup_{i=1}^m A_i$;
  \item $\tau(O) = \sup\{\tau(K)\,|\,K \in \cC(X),\,K \subset O\}$ for every $O \in \cO(X)$.
\end{enumerate}
We say that $\tau$ is simple if it only takes values $0,1$.
\end{defin}
\begin{rem}
Note that the restriction to $\cA(X)$ of a Borel probability measure on $X$ is a topological measure, which is simple if and only if it is a delta-measure.
\end{rem}

The Aarnes representation theorem proved in \cite{Aarnes_Qss_qms} says that quasi-states and topological measures are in a $1$-to-$1$ correspondence, as follows. If $\zeta$ is a quasi-state on $X$, the corresponding topological measure $\tau \fc \cA(X) \to \R$ is defined for $K \in \cC(X)$ by
$$\tau(K) = \inf\{\zeta(f)\,|\,f \in C(X)\,,f\geq \chi_K\}$$
and for $O \in \cO(X)$ by $\tau(O) = 1 - \tau(X - O)$, where $\chi_K$ is the indicator function of $K$. Conversely, if $\tau$ is a topological measure on $X$, the corresponding quasi-state $\zeta$ is given by
\begin{equation}\label{eq:qss_from_tm_via_integration}
\zeta(f) = \min_X f + \int_{\min_X f}^{\max_X f}\tau(\{f\geq s\})\,ds
\end{equation}
for $f \in C(X)$. The Aarnes representation theorem extends the usual Riesz representation theorem, in the sense that the restriction of a Borel probability measure $\mu$ to $\cA(X)$ corresponds by the above bijection to the integral $\int_X\cdot\,d\mu$. Moreover, simple quasi-states correspond to simple topological measures.

\subsubsection{Aarnes's construction}\label{sss:Aarnes_construction_Psi}

In \cite{Aarnes_Construction_non_subadd_meas_discretization_Borel_meas} Aarnes describes a general method of constructing simple topological measures on spaces satisfying a condition which Knudsen showed in \cite{Knudsen_Topology_construction_extreme_qms} to be equivalent to $H^1(X;\Z) = 0$ if $X$ is a CW complex. Since we are only concerned with smooth manifolds, and any compact manifold is homeomorphic to a finite CW complex, for instance via a triangulation \cite{Cairns_A_simple_triangulation_method_smooth_mfds}, this latter condition is the one we will use.

Aarnes's method uses so-called solid sets. A set $A \subset X$ is solid if it is connected and its complement is also connected. To use Aarnes's method, we need the notion of the spectrum of a Borel probability measure (see \cite{Butler_q_fcns_extreme_tms}, where it is called the split spectrum). Let $\cP(X)$ denote the space of Borel probability measures on $X$. For $\mu \in \cP(X)$ its spectrum is the set $\Spec(\mu)$ of numbers $\alpha \in (0,1)$ such that there are disjoint solid $C,C' \in \cC(X)$ with $\mu(C) = \alpha$, $\mu(C') = 1-\alpha$. We then have \cite{Aarnes_Construction_non_subadd_meas_discretization_Borel_meas}: if $\mu \in \cP(X)$ is such that $\frac 1 2 \notin \Spec(\mu)$, then there is a unique simple topological measure $\tau_\mu$, such that for solid $C \in \cC(X)$ we have
\begin{equation}\label{eq:defin_tm_corresp_to_meas}
\tau_\mu(C) = \left\{\begin{array}{ll}0 \,, & \text{if }\mu(C) < \frac 1 2 \\ 1 \,, & \text{if } \mu(C) \geq \frac 1 2\end{array}\right.
\end{equation}
For future use we note that if $O \in \cO(X)$ is solid, then $\tau_\mu(O) = 0$ if $\mu(O) \leq \frac 1 2$ and $\tau_\mu(O) = 1$ otherwise.

Denoting
$$\cP_0(X) = \{\mu \in \cP(X)\,|\, \tfrac 1 2 \notin \Spec(\mu)\}\,,$$
we therefore have the map announced at the beginning of this section:
$$\Psi \fc \cP_0(X) \to \cQ(X)\,, \qquad \mu \mapsto \zeta_\mu\,,$$
where $\cQ(X)$ is the set of quasi-states on $X$, and $\zeta_\mu$ is the quasi-state corresponding to $\tau_\mu$ by the Aarnes representation theorem.

\subsubsection{Continuity of $\Psi$}\label{sss:continuity_Psi}

Here we state our main result concerning the metric continuity property of $\Psi$, which is used in Section \ref{ss:pf_main_thm} to prove the estimate \eqref{eq:estimate_main_algorithm}. To this end we need certain natural metrics on the spaces of measures and quasi-states in case $X$ is a metric space. Let us therefore assume that $(X,d)$ is a compact metric space. Wasserstein metrics are a natural family of metrics on $\cP(X)$ parametrized by $p \in [1,\infty]$. They appear in optimal transport theory, see \cite{Villani_Optimal_transport}. For $p\in[1,\infty)$ we have the $p$-Wasserstein metric
$$W_p(\mu,\nu) = \inf_{\rho} \left(\int_{X\times X}d(x,y)\,d\rho(x,y)\right)^{1/p}\,,$$
where $\rho$ runs over Borel probability measures on $X \times X$ with marginals $\mu$ and $\nu$, that is $p_*\rho = \mu$ and $q_*\rho = \nu$ where $p,q\fc X \times X \to X$ are the two projections.
The $\infty$-Wasserstein metric is defined by
\begin{equation}\label{eq:W_infty_metric}
W_\infty(\mu,\nu) = \inf_{\rho}\max_{(x,y)\in\supp\rho}d(x,y)\,,
\end{equation}
with $\rho$ as before. The Kantorovich--Rubinstein duality \cite{Kantorovich_Rubinstein_On_a_space_of_completely_additive_functions}, \cite{Fernique_Sur_le_theoreme_Kantorovich_Rubinstein_espaces_Polonais}, \cite{Villani_Optimal_transport} expresses $W_1$ via integration of Lipschitz functions on $X$. Namely, let us denote by $\Lip(X,d)$ the space of Lipschitz functions on $X$ and for $f \in \Lip(X,d)$ its Lipschitz pseudo-norm
$$\|f\|_{\Lip} = \inf\{L \geq 0\,|\,|f(x) - f(y)| \leq Ld(x,y)\text{ for all }x,y\in X\}\,.$$
Then
\begin{equation}\label{eq:W_1_dual_defin_via_Lip_fcns}
W_1(\mu,\nu) = \sup\left\{\left|\int_Xf\,d\mu - \int_X f\,d\nu\right|\,\Big|\,f \in \Lip(X,d)\,,\|f\|_{\Lip} \leq 1\right\}\,.
\end{equation}

It is tempting to try to define analogous metrics on $\cQ(X)$, but the major issue is that constructing quasi-states on $X \times X$ with given marginals is quite a nontrivial task, see \cite{Grubb_products_qms}. However the dual definition of $W_1$, suggested to us by Leonid Polterovich, admits a straightforward generalization to $\cQ(X)$:
$$W_1(\zeta,\zeta') = \sup\left\{|\zeta(f) - \zeta'(f)|\,\big|\,f \in \Lip(X,d)\,,\|f\|_{\Lip} \leq 1\right\}\,,$$
and, as it turns out, this is exactly the metric we need:
\begin{thm}\label{thm:metric_conv_qss}
Let $X$ be a closed connected manifold with $H^1(X;\Z) = 0$ and let $d$ be a metric on $X$ inducing its topology. Then the map
$$\Psi \fc (\cP_0(X),W_\infty) \to (\cQ(X),W_1)$$
is $1$-Lipschitz, that is for all $f \in \Lip(X,d)$ and $\mu,\nu \in \cP_0(X)$ we have
$$|\zeta_\mu(f) - \zeta_\nu(f)| \leq \|f\|_{\Lip}\cdot W_\infty(\mu,\nu)\,.$$
\end{thm}

For completeness we include the following continuity result, which does not use metrics.
\begin{thm}\label{thm:weak_conv_qss}Assume $X$ is a closed connected manifold with $H^1(X;\Z) = 0$. Then the map $\Psi$ is continuous, where $\cP_0(X)$ and $\cQ(X)$ are given the weak topology.
\end{thm}
\noindent Theorems \ref{thm:metric_conv_qss}, \ref{thm:weak_conv_qss} are proved in Section \ref{ss:approx_results}. It is interesting to note that they do not simply follow from one another, even though their proofs use similar ideas. Let us describe the reason in more detail. It is well-known that the $p$-Wasserstein distances for finite $p$ induce the weak topology on $\cP(X)$ \cite{Villani_Optimal_transport}. Similarly we have the following proposition, proved in Section \ref{ss:approx_results}:
\begin{prop}\label{prop:W_1_induces_weak_top_qss}If $(X,d)$ is a compact metric space, then the metric $W_1$ on $\cQ(X)$ induces the weak topology.
\end{prop}
\noindent However the topology induced on $\cP(X)$ by $W_\infty$ is strictly stronger than the weak topology, for instance on $[0,1]$ we have $W_1(\frac{1}{n}\delta_0 + \frac{n-1}{n}\delta_1,\delta_1)\to 0$ as $n\to\infty$, while $W_\infty(\frac{1}{n}\delta_0 + \frac{n-1}{n}\delta_1,\delta_1) = 1$ for all $n$. For this reason we cannot simply say that Theorem \ref{thm:weak_conv_qss} follows from Theorem \ref{thm:metric_conv_qss}, since the topologies on $\cP_0(X)$ appearing in the two theorems are different. In the opposite direction, Theorem \ref{thm:weak_conv_qss} implies that $\Psi$ is continuous with respect to the weak topology on $\cQ(X)$ and the topology induced on $\cP_0(X)$ by $W_\infty$, but it does not imply the Lipschitz property.

\subsection{Non-approximation results}\label{ss:non_approx}

In this section we present a wide class of quasi-states which \emph{cannot} be approximated by quasi-states of the form $\Psi(\mu)$ for $\mu \in \cP_0(X)$. The main application is the fact that symplectic quasi-states constructed via Floer homology on symplectic manifolds of dimension at least $4$ cannot be approximated by such quasi-states.

We have the following general non-approximation result for quasi-states ``supported'' on submanifolds of codimension at least $2$.
\begin{thm}\label{thm:general_non_approximation}
Let $X$ be a closed connected Riemannian manifold with $H^1(X;\Z) = 0$, and let $d$ be the induced distance function. Let $\zeta$ and $\tau$ be a quasi-state and a topological measure on $X$, corresponding to one another by the Aarnes representation theorem. Assume that there are closed connected submanifolds $Y_j$, $j=1,\dots,4$ of codimension at least $2$, whose triple intersections are all empty, such that $\tau(Y_j)=1$ for all $j$. Then there is a constant $c > 0$ such that for all $\mu \in \cP_0(X)$ we have
$$W_1(\zeta,\zeta_\mu) \geq c\,.$$
\end{thm}
\noindent In particular, $\zeta$ cannot be weakly approximated by quasi-states of the form $\zeta_\mu$ for $\mu \in \cP_0(X)$, thanks to Proposition \ref{prop:W_1_induces_weak_top_qss}.

A typical application of this theorem is in a situation when $\zeta$ is invariant under a sufficiently large symmetry group and when there is a closed connected submanifold $Y \subset X$ of codimension at least $2$ with $\tau (Y) = 1$. Here the result holds if there are symmetries $g_1,g_2,g_3$ of $\zeta$ such that the quadruple of sets $Y,g_1(Y),g_2(Y),g_3(Y)$ has empty triple intersections. What follows is a general result of this type in the context of symplectic geometry.
\begin{thm}\label{thm:non_approx_sympl_mfds}
Let $(X,\omega)$ be a closed connected symplectic manifold of dimension $\geq 4$ such that $H^1(X;\Z) = 0$, let $\zeta$ be a quasi-state invariant under the action of the Hamiltonian group $\Ham(X,\omega)$, and assume that there is a closed connected Lagrangian $L \subset X$ with the property that $\tau(L) = 1$, where $\tau$ is the topological measure associated to $\zeta$. Then, given a Riemannian metric on $X$, there is a constant $c > 0$ such that
$$W_1(\zeta,\zeta_\mu) \geq c$$
for all $\mu \in \cP_0(X)$.
\end{thm}
\begin{rem}
The assumption on the dimension is essential. Indeed, if $\zeta$ is the median quasi-state on $S^2$ from Example \ref{exam:median_Aarnes_qss} and $\tau$ is the corresponding topological measure, then $\zeta$ is invariant under $\Ham(S^2,\omega)$, where $\omega$ is the standard round symplectic form. A great circle $L \subset S^2$ is a Lagrangian submanifold with $\tau(L) = 1$. Of course, $\zeta$ is the value of $\Psi$ at the Lebesgue measure.
\end{rem}

The following corollary is the application of interest to us. Here $\zeta$ refers to symplectic quasi-states constructed in \cite{Entov_Polterovich_Calabi_quasimorphism_quantum_homology}.
\begin{coroll}\label{coroll:non_approx_CPn_S2_x_S2}
Let $X$ be $\C P^n$, $n\geq 2$, or $S^2 \times S^2$, endowed with the standard K\"ahler structure. Then there is a constant $c > 0$ such that for all $\mu \in \cP_0(X)$ we have
$$W_1(\zeta,\zeta_\mu) \geq c\,.$$
\end{coroll}
\noindent This result is interesting, because quasi-states of the form $\zeta_\mu$, $\mu \in \cP_0(X)$ are constructed using ``soft'' techniques, referring in symplectic topology to results which do not use elliptic PDEs such as pseudoholomorphic curves (see \cite{Gromov_Pseudo_holo_curves_sympl_mfds}). On the other hand the quasi-states $\zeta$ are constructed using ``hard'' techniques of symplectic topology, and it would be indeed very surprising if $\zeta$ could be obtained as a weak limit of the $\zeta_\mu$.

\begin{rem}
There are many more examples for which Corollary \ref{coroll:non_approx_CPn_S2_x_S2} holds, see \cite{Entov_Polterovich_rigid_subsets_sympl_mfds}, however for the sake of simplicity we only mention these two examples, which suffice to illustrate our point.
\end{rem}

We do not know the optimal value of the constant $c$ in Corollary \ref{coroll:non_approx_CPn_S2_x_S2}, which prompts the following
\begin{question}
What is the optimal value of the constant $c$?
\end{question}
\noindent An explicit sharp value would be worthwhile to compute.

\begin{acknow}We would like to thank Leonid Polterovich for suggesting the problems which gave rise to the results of this paper, for his constant interest, and for numerous discussions. We also thank Asaf Kislev, Pazit Haim-Kislev, and Eyal Ackerman for discussions regarding algorithms computing Reeb graphs of PL functions and the quasi-states corresponding to discrete measures. A.\ D.\ is partially supported by the Israel Science Foundation grant 1380/13. F.\ Z.\ is partially supported by the Israel Science Foundation grant 1825/14, and by grant number 1281 from the GIF, the German--Israeli Foundation for Scientific Research and Development.
\end{acknow}

\section{Proofs}\label{s:proofs}

\subsection{Preliminaries}\label{ss:prelims}

\subsubsection{Simple quasi-states on manifolds}\label{sss:simples_qss_mfds}

Theorem \ref{thm:main_result} is concerned with the median quasi-state, which is a simple quasi-state on $S^2$. Here we describe a general method of computing simple quasi-states on manifolds, in particular obtaining the descriptions of Example \ref{exam:median_Aarnes_qss}. The material here is not new. We gather it here for completeness and to establish notation.

Let $Y$ be a compact tree, that is a finite simply-connected $1$-dimensional CW complex. The result of \cite{Wheeler_Quasi_meas_dim_thry} says that a quasi-state on $Y$ must be linear. If $\nu \in \cP_0(Y)$, it follows that $\Psi(\nu) \in \cQ(Y)$, being a simple quasi-state on $Y$, is the evaluation at a point $y_0 \in Y$ by Remark \ref{rem:elementary_pties_qss}. The definition of $\Psi(\nu)$ in this case says that $y_0$ is the unique point with the property that every connected component $C$ of $Y - \{y_0\}$ satisfies $\nu(C) \leq \frac 1 2$. Indeed, we have $\tau_\nu(\{y_0\}) = 1$, therefore, if $C_1,\dots,C_k$ are the connected components of $Y - \{y_0\}$, we have $\tau_\nu(C_j) = 0 $ for all $j$, and since $C_j$ are all open solid sets, it follows that $\nu(C_j) \leq \frac 1 2$. The uniqueness follows from the assumption that $\frac 1 2$ does not belong to the spectrum of $\nu$.

Let now $X$ be a closed connected manifold with $H^1(X;\Z) = 0$ and let $F \in C^\infty(X)$ be a Morse function. Let $\Gamma_F$ be the Reeb graph of $F$ \cite{Reeb_Sur_points_singuliers_forme_Pfaff_completement_integrable_fonc_numerique}, that is the quotient space of $X$ by the equivalence relation whose equivalence classes are the connected components of level sets of $F$. Let $\pi \fc X \to \Gamma_F$ be the quotient map. It is well-known that $\pi$ is surjective on fundamental groups, which implies that $\Gamma_F$ is a tree. Let $\mu \in \cP_0(X)$. It can be checked that if $C \subset \Gamma_F$ is closed and solid, then so is $\pi^{-1}(C)$, which implies that $\pi_*\mu \in \cP_0(\Gamma_F)$, and moreover that $\pi_*\tau_\mu = \tau_{\pi_*\mu}$. It follows that $\Psi(\pi_*\mu) = \pi_*\Psi(\mu)$, meaning $\Psi(\pi_*\mu)(g) = \Psi(\mu)(g\circ \pi)$ for $g \in C(\Gamma_F)$.

The above considerations then imply that $\Psi(\pi_*\mu) \in \cQ(\Gamma_F)$ is the evaluation at a point, which we refer to as the $\mu$-median of $F$, denoted $m_\mu(F) \in \Gamma_F$. In particular, if $\ol F \in C(\Gamma_F)$ is such that $F = \ol F \circ \pi$, then
$$\ol F(m_\mu(F)) = \Psi(\pi_*\mu)(\ol F) = (\pi_*\Psi(\mu))(\ol F) = \Psi(\mu)(\ol F \circ \pi) = \Psi(\mu)(F)\,.$$
Conclusion: the value of the simple quasi-state $\Psi(\mu)$ on the Morse function $F$ is the value of $\ol F$ at the median $m_\mu(F) \in \Gamma_F$, which is the unique point of $\Gamma_F$ with the property that every connected component $C$ of $\Gamma_F - \{m_\mu(F)\}$ is such that $\pi_*\mu(C) \leq \frac 1 2$. Stated in terms of $X$ and $F$ only, there is a unique component $m_\mu(F)$ of a level set of $F$ with the property that every connected component $C$ of $X - m_\mu(F)$ satisfies $\mu(C) \leq \frac 1 2$, and $\Psi(\mu)(F)$ is the value of $F$ on $m_\mu(F)$.

The same reasoning applies when $X$ is a finite polyhedron with $H^1(X;\Z) = 0$ and $F \in C(X)$ is a piecewise linear function such that its values at the vertices of $X$ are pairwise distinct, meaning that for $\mu \in \cP_0(X)$ we can define the Reeb tree $\Gamma_F$ of $F$, the $\mu$-median $m_\mu(F)$ and then $\Psi(\mu)(F)$ is the value of $F$ on $m_\mu(F)$. We will use this particular case in the proof of Theorem \ref{thm:main_result}, using a triangulation of $S^2$ and the corresponding piecewise linear approximation of a given function.

\subsubsection{$W_\infty$ estimate}\label{sss:W_infty_estimate}

The metric continuity result, Theorem \ref{thm:metric_conv_qss}, will be used in the proof of the estimate \eqref{eq:estimate_main_algorithm} asserted in Theorem \ref{thm:main_result}. It will be applied in the following form.
\begin{prop}\label{prop:discrete_meas_close_to_given_one_W_infty}
Let $(X,d)$ be a compact metric space, let $X = \bigcup_{i=1}^kA_i$ be a measurable partition, and let $\mu_0 \in \cP(X)$. Pick $x_i \in A_i$ for all $i$ and define
$$\mu = \sum_{i=1}^{k}\mu_0(A_i)\cdot \delta_{x_i}\,.$$
Then
$$W_\infty(\mu_0,\mu) \leq \max_{1\leq i \leq k}\diam \ol{A_i}\,.$$
\end{prop}
\begin{proof}For a measurable set $A \subset X$ define $\mu_0|_A$ by $\mu_0|_A(B) = \mu_0(A\cap B)$ for measurable $B \subset X$. The probability measure
$$\rho = \sum_{i=1}^{k}\mu_0|_{A_i}\otimes \delta_{x_i}\,,$$
has $\mu_0,\mu$ as its marginals. A point $(x,y)$ lies in the support of $\rho$ if an only if there is $i$ such that $y = x_i$ and $x \in \ol{A_i}$, thus $d(x,y) \leq \diam \ol{A_i}$ and the result follows from the definition of $W_\infty$, \eqref{eq:W_infty_metric}.
\end{proof}

For the proof of estimate \eqref{eq:estimate_main_algorithm} we need the following particular case. We will subdivide $S^2$ into $k$ regions of equal area and Euclidean diameter $\leq 7/\sqrt k$ as in \cite{Zhou_Arrangements_of_points_on_the_sphere}, choose a point in the interior of every region, and assemble these points into a set $Z$. Letting $\mu$ be the Lebesgue measure and $\mu_Z$ be the probability measure uniformly distributed on $Z$, Proposition \ref{prop:discrete_meas_close_to_given_one_W_infty} in this case says that
\begin{equation}\label{eq:W_infty_estimate_subdivision}
W_\infty(\mu,\mu_Z) \leq \frac{7}{\sqrt k}\,.
\end{equation}

\subsection{Proof of Theorem \ref{thm:main_result}}\label{ss:pf_main_thm}

The proof consists of a description of the announced algorithm, and a verification of the estimate \eqref{eq:estimate_main_algorithm}. Before we move on to the details, let us give an overview of the algorithm as well as the estimate. We remind the reader that the metric on $S^2$ we are using is the restriction of the Euclidean metric on $\R^3$, and we denote it by $d$.

\subsubsection{Overview}\label{sss:overview}

First, given the parameter $N \geq 46$, we construct a triangulation $\cT_N$ of $S^2$ having $20N^2$ triangles, using a regular inscribed icosahedron. The lower bound on $N$ is explained below. We replace the given function $f$ by the unique function $F$ which is piecewise linear on the simplexes of $\cT_N$ and which coincides with $f$ at the vertices of $\cT_N$. The triangulations we use imply that
\begin{equation}\label{eq:estimates_f_F_overview}
\|f - F\|_{C^0} \leq \|f\|_{\Lip}\cdot \frac{\sqrt 3(3-\sqrt 5)}{N} \quad \text{and} \quad \|F\|_{\Lip} \leq \frac{13+6\sqrt 5}{11}\,\pi\cdot\|f\|_{\Lip}\,.
\end{equation}
Note the numerical values of these constants: $1.323\dots$ and $7.545\dots$ respectively. The second constant especially is most likely far from the actual values obtained in practice; here we present constants for which we can obtain relatively simple proofs.

Since by Remark \ref{rem:elementary_pties_qss}, quasi-states are $1$-Lipschitz with respect to the $C^0$-norm, we have
$$|\zeta(f) - \zeta(F)| \leq \|f\|_{\Lip}\cdot\frac{\sqrt 3(3-\sqrt 5)}{N}\,.$$

The next step is to approximate $\zeta(F)$. This is done as follows. We subdivide $S^2$ into $k$ regions of equal area and diameter $\leq 7/\sqrt k$, where $k$ is an odd integer $\geq 237$. We assume this bound for the estimates in Lemmata \ref{lemma:complexity_computing_subdivision}, \ref{lemma:complexity_finding_region_contain_given_point}, \ref{lemma:lower_bound_radius_spher_cap_region} to work. Moreover, we assume that $k \leq 0.115744\cdot N^2$, which in particular forces $N \geq 46$. This assumption and the specific subdivision we use guarantees that each region contains a vertex of $\cT_N$ in its interior. We fix such a vertex for every region, and let $Z \subset S^2$ be the set of these vertices. The idea now is to replace $\zeta(F)$ by $\zeta_Z(F)$, where $\zeta_Z$ is the Aarnes quasi-state corresponding to $Z$, see Example \ref{exam:median_Aarnes_qss}. Equation \eqref{eq:W_infty_estimate_subdivision} says that $W_\infty(\mu,\mu_Z) \leq 7/\sqrt k$, and therefore Theorem \ref{thm:metric_conv_qss} yields $W_1(\zeta,\zeta_Z) \leq 7/\sqrt k$, whence
$$|\zeta(F) - \zeta_Z(F)| \leq \|F\|_{\Lip}\cdot \frac 7{\sqrt k}\,,$$
therefore in total we obtain
$$|\zeta(f) - \zeta_Z(F)| \leq \|f\|_{\Lip}\cdot \left(\frac{\sqrt 3(3-\sqrt 5)}{N} + \frac{7\pi(13+6\sqrt 5)}{11}\frac{1}{\sqrt k}\right)\,,$$
or, taking the maximal value of $k$ permitted by the above assumption, which implies $k \geq 0.115744\cdot N^2 - 2$, we have
$$|\zeta(f) - \zeta_Z(F)| \leq \|f\|_{\Lip}\cdot \frac{C_8}{N}\,,$$
where $C_8 \approx 197.778\dots$, which is the estimate asserted in Theorem \ref{thm:main_result}.

It remains to compute $\zeta_Z(F)$, which will be the output of the algorithm. We can perturb $F$ such that it has pairwise distinct values at the vertices of $\cT_N$. Such a perturbation can be made arbitrarily small in the $C^0$-norm, and thus its effect on $\zeta_Z(F)$ is arbitrarily small. It follows from the discussion in Section \ref{sss:simples_qss_mfds} that $\zeta_Z(F)$ is obtained as follows. Let $\pi \fc S^2 \to \Gamma_F$ be the quotient map onto the Reeb tree of $F$. The restriction of $\pi$ to the set of vertices of $\cT_N$ is injective by the definition of the Reeb graph, therefore $\pi_*\mu_Z$ is the uniformly distributed discrete measure supported on the set $\pi(Z)$. The $\mu_Z$-median $m_Z(F)$ of $F$ is the unique point in $\Gamma_F$ such that every connected component $C$ of $\Gamma_F - \{m_Z(F)\}$ satisfies $(\pi_*\mu_Z)(C) \leq \frac 1 2$, meaning that $C$ contains less than half the points of the set $\pi(Z)$. We use an algorithm described in \cite{Edelsbrunner_et_al_Loops_Reeb_graphs_2_mfds} to compute $\Gamma_F$ as a rooted tree whose nodes are in bijection with the vertices of $\cT_N$. It follows that $m_Z(F)$ must be a node of $\Gamma_F$, and the remaining step of the algorithm finds it via enumeration.

\subsubsection{The algorithm}\label{sss:algorithm}

Here we present a detailed description of the algorithm summarized above, together with a computation of its complexity. The algorithm is comprised of several steps. From the point of view of complexity, the most demanding step, and indeed the step whose complexity dominates that of the other ones, is the computation of the Reeb tree of the piecewise linear function $F$. There are a few auxiliary results, whose proofs we defer to Section \ref{sss:pfs_aux_results}.

Fix parameters $N,k$ satisfying the assumptions in Section \ref{sss:overview}.

\tb{Step 1.} We subdivide $S^2$ into an odd number $k\geq 237$ of regions having equal area and diameter $\leq 7/\sqrt k$, using the algorithm described in \cite{Zhou_Arrangements_of_points_on_the_sphere}. To every region $R$ we associate a boolean flag $f_R$, initially set to false.
\begin{lemma}\label{lemma:complexity_computing_subdivision}
The complexity of this step if $O(k)$.
\end{lemma}

\tb{Step 2.} We construct a triangulation $\cT_N$ of $S^2$ into $20N^2$ triangles, as follows. Take a regular icosahedron inscribed into $S^2$, subdivide each one of its faces into $N^2$ congruent triangles, and project them to $S^2$ via the radial projection. For each vertex $v$ of $\cT_N$ we create a boolean flag $c_v$, initially set to false. Then we enumerate the vertices, and for each vertex $v$ we compute the region $R$ containing $v$ in its interior, and if $f_R$ is set to true, we move on to the next vertex, otherwise we set $f_R$ and $c_v$ to true.

We have the following important result. A spherical cap is the intersection of $S^2$ with a half-space of $\R^3$.
\begin{lemma}\label{lemma:lower_bound_radius_spher_cap_region}
Every region of the subdivision of Step 1 contains a spherical cap for which the Euclidean distance from its center to the rim is at least
$$\frac{\sqrt 3(3-\sqrt 5)}{N}\,.$$
\end{lemma}
Together with the assumption on $N,k$, this implies that the algorithm results in a choice of a unique vertex of $\cT_N$ for every region $R$, thereby marking $k$ vertices.

The particular algorithm from Step 1 which computes the regions allows for the following.
\begin{lemma}\label{lemma:complexity_finding_region_contain_given_point}
The complexity of finding the region containing a given point in $S^2$ is $O(1)$.
\end{lemma}

Since the enumeration has complexity $O(N^2)$ times the maximal complexity of each step, and the latter is $O(1)$ as implied by the lemma, we see that the total complexity for this step is $O(N^2)$.

\tb{Step 3.} Replace $f$ with the PL function $F$ whose values at the vertices of the triangulation coincide with those of $f$. This has complexity $O(N^2)$.

\tb{Step 4.} Using the algorithm described in \cite{Edelsbrunner_et_al_Loops_Reeb_graphs_2_mfds}, we compute the Reeb tree $\Gamma_F$ of $F$. This algorithm runs in time $O(N^2 \log N^2) = O(N^2 \log N)$, see \emph{ibid}. Its output is a rooted tree with the root being the absolute minimum of $F$, and the nodes being in a one-to-one correspondence with the vertices of $\cT_N$. The total complexity of this step is $O(N^2\log N)$.

\tb{Step 5.} Label the nodes $v$ of the Reeb tree with the flags $c_v$ of Step 2, and treat them as integers, so that false corresponds to $0$ while true corresponds to $1$. Create an integer counter $t_v$ for each node $v$ of the tree and set it to zero. Define the following recursive function COUNT which takes as input a node of $\Gamma_F$ and whose output is a nonnegative integer: the function, applied to a node $v_0$ returns the sum of $c_{v_0}$ (treated as an integer) and the return values of COUNT applied to all the children of $v_0$ if it has any; also COUNT records the integer thus obtained in the counter $t_{v_0}$.

After running COUNT on the Reeb tree of $F$ every node $v_0$ carries the number of vertices marked as true by the flags $c_v$ where $v$ runs over all the descendants of $v_0$ including $v_0$ itself.

This step has complexity $O(N^2)$.

\tb{Step 6.} The last step computes the $Z$-median $m_Z(F)$ of $F$. We remarked above that it is one of the nodes of $\Gamma_F$. Indeed, since the measure $\pi_*\mu_Z$ is concentrated on nodes, $m_Z(F)$ cannot lie in the interior of an edge, because it would violate uniqueness. We claim that the desired node $m_Z(F)$ is the unique one satisfying the following conditions: (i) its count $t_{m_Z(F)}$ from the previous step is at least $\lceil \frac k 2\rceil$ and (ii) it has maximal depth with respect to the root of $\Gamma_F$. The uniqueness can be seen as follows. If $v,v'$ are two such nodes, they have the same root depth, therefore the subtrees rooted at $v,v'$ must be disjoint, since otherwise they would coincide. But then the total count of $v$ and $v'$ would be at least $2\lceil \frac k 2\rceil > k$, contradicting the fact that $\pi(Z)$ contains $k$ points. It is also clear that such a node exists since the count of the root of $\Gamma_F$ is $k \geq \lceil \frac k 2\rceil$, and it drops as we move away from the root.

It follows that $m_Z(F)$ can be found by enumerating the nodes of $\Gamma_F$. This has complexity $O(N^2)$.

\tb{Total complexity of the algorithm:} We see that the total complexity of the algorithm is $O(N^2\log N)$, as claimed.

\subsubsection{The estimate}\label{sss:estimate}

It remains to prove the inequalities \eqref{eq:estimates_f_F_overview}, which is what we do here. Recall that we are using the restriction of the Euclidean metric to $S^2$.

First, we treat the case of general triangulations. By a triangulation we mean a finite collection of points in $S^2$, arranged into triples, corresponding to the vertices of the triangles of the triangulation, such that the following holds:
\begin{itemize}
  \item if $(z_0,z_1,z_2)$ is such a triple, then the corresponding linear simplex, which is the convex hull of the $z_i$ in $\R^3$, does not pass through $0 \in \R^3$;
  \item if $r \fc \R^3-\{0\} \to S^2$ is the radial projection, the images $r(\sigma)$ cover $S^2$ as $\sigma$ varies over the linear simplexes corresponding to the triangles of the triangulation, while their interiors are pairwise disjoint;
  \item if two triangles $\sigma,\sigma'$ of the triangulation intersect, they either coincide, or the intersection is a common vertex or a common edge of $\sigma,\sigma'$.
\end{itemize}
In this case there is a unique continuous function $F \in C(S^2)$ such that $F|_{r(\sigma)} = \wt F_\sigma \circ r_\sigma^{-1}$ for every linear simplex $\sigma$, where $r_\sigma=r|_\sigma$ and $\wt F_\sigma \fc \sigma \to \R$ is the unique affine function which coincides with $f$ at the vertices $z_0,z_1,z_2$ of $\sigma$. We have
\begin{prop}\label{prop:pties_PL_approx_sphere}
\begin{enumerate}
\item $\|f - F\|_{C^0} \leq \|f\|_{\Lip}\cdot \max_\sigma\diam r(\sigma)$, where $\sigma$ runs over the linear simplexes corresponding to the triangulation;
\item $F$ is Lipschitz with respect to $d$ and
\begin{equation}\label{eq:Lip_cst_PL_fcn}
\|F\|_{\Lip} \leq \frac{\pi}{2\sin(\theta_0/2)}\cdot\|f\|_{\Lip}\,,
\end{equation}
where $\theta_0 \in (0,\frac \pi 2]$ is such that all the linear simplexes of the triangulation have the property that each one of their angles lies in $[\theta_0,\pi - \theta_0]$.
\end{enumerate}
\end{prop}

To prove the estimates \eqref{eq:estimates_f_F_overview}, we need to prove the following lemma. By a curvilinear simplex we mean the image $r(\sigma) \subset S^2$ of a linear simplex $\sigma$ of the triangulation. Recall the icosahedral triangulation from Step 2 of Section \ref{sss:algorithm}.
\begin{lemma}\label{lemma:estimates_icosahedral_triangulation}
The icosahedral triangulation with $20N^2$ triangles satisfies: (i) the maximal diameter of a curvilinear simplex is at most
$$\frac{\sqrt 3 (3-\sqrt 5)}N\,,$$
(ii) the angle $\theta_0$ of point (ii) of Proposition \ref{prop:pties_PL_approx_sphere} can be chosen to equal
$$\pi-2\arccos\left(\frac 1 2 \left(6\sqrt 5 - 13\right)\right)\,,$$
therefore the constant in equation \eqref{eq:Lip_cst_PL_fcn} is
$$\frac{\pi}{2\sin(\theta_0/2)} = \frac{13+6\sqrt 5}{11}\,\pi\,.$$
\end{lemma}

\subsubsection{Proofs of auxiliary results}\label{sss:pfs_aux_results}

Here we give the proofs of Lemmata \ref{lemma:complexity_computing_subdivision}, \ref{lemma:lower_bound_radius_spher_cap_region}, \ref{lemma:complexity_finding_region_contain_given_point}, \ref{lemma:estimates_icosahedral_triangulation}, and Proposition \ref{prop:pties_PL_approx_sphere}.

\begin{proof}[Proof of Lemma \ref{lemma:lower_bound_radius_spher_cap_region}] Let us summarize the necessary information from \cite{Zhou_Arrangements_of_points_on_the_sphere}. The result we need from there is Theorem 2.8, whose proof starts on page 25. For concreteness, we fix the parameters $N_0,k_0,r$ on that page to their sample values $N_0 = 237, k_0 = 6, r = 2$ as stated on page 27 after equation (2.4). What is written there is that the estimates based on the values of the parameters hold when the number of regions of the subdivision is at least $N_0$. Moreover, the number of regions should be at least $237$ for the estimates to work. This is the ultimate reason for our assumptions on $N,k$.

We denote the spherical metric on $S^2$ by $D$ throughout, that is $\cos D(x,y) = \langle x,y\rangle$ for $x,y \in S^3$.

Points on $S^2$ will be denoted here using their spherical coordinates $(\theta,\phi)$. The regions of the subdivision are given by rectangles in spherical coordinated $(\theta,\phi) \in [0,\pi] \times [0,2\pi]$. The construction described in the proof of Theorem 2.8 \emph{ibid.} results in the subdivision of $S^2$ into bands bounded by parallels with latitudes
$$\theta_{-1} = 0\,, \theta_0 = \arccos\left(1-\frac{12}{k}\right)\,,\dots\,,\theta_n = \pi - \theta_0\,,\theta_{n+1} = \pi\,,$$
where $n$ is the largest odd integer $\leq \sqrt{k/2}$. The band bounded between $\theta_{i-1}$ and $\theta_i$ contains $m_i$ regions, each one given in spherical coordinates by the rectangle
$$D_{ij} = [\theta_{i-1},\theta_i]\times\left[\frac{2\pi(j-1)}{m_i},\frac{2\pi j}{m_i}\right]\,,\quad i = 0,\dots,n+1\,,j=1,\dots,m_i\,.$$
We have $m_0 = m_{n+1} = 6$.

Fix an integer $1\leq i \leq n$. Consider the center of $D_{ij}$ in spherical coordinates:
$$\theta = \frac{\theta_{i-1} + \theta_i}{2}\,,\quad \phi = \frac{2\pi(j-1)}{m_i}\,.$$
It is not hard to see that the minimal spherical distance from $(\theta,\phi)$ to a parallel passing through the top or the bottom edge of $D_{ij}$ is $\frac{\theta_i - \theta_{i-1}}2$. Elementary spherical geometry also shows that the minimal spherical distance from $(\theta,\phi)$ to a meridian passing through the right or the left edge of $D_{ij}$ is
$$\arccos\sqrt{\cos^2\theta+\sin^2\theta\cos^2\frac{\pi}{m_i}}\,.$$
It follows that the minimal spherical distance from $(\theta,\phi)$ to a point in the boundary $\partial D_{ij}$ is
$$D((\theta,\phi),\partial D_{ij}) \geq \min\left(\frac{\theta_i-\theta_{i-1}}{2}\,,\arccos\sqrt{\cos^2\theta+\sin^2\theta\cos^2\frac{\pi}{m_i}}\right)\,.$$
Using the inequalities $\arccos x \geq \sqrt{1-x^2}$ and $\sin t \geq \frac 2 \pi t$ for $t \in [0,\pi/2]$, this implies
$$D((\theta,\phi),\partial D_{ij}) \geq \min\left(\frac{\theta_i - \theta_{i-1}}{2}\,,\frac{2}{m_i}\,\sin\frac{\theta_i + \theta_{i-1}}{2}\right)\,.$$
Let us estimate the second term in the minimum. For this we use the known fact that the area between parallels with latitudes $\theta_{i-1},\theta_i$ is $2\pi(\cos\theta_{i-1} - \cos\theta_i)$. On the other hand, since this band contains $m_i$ regions, it has area $\frac{4\pi m_i}{k}$, meaning
$$m_i = \frac{k}{2}(\cos\theta_{i-1} - \cos\theta_i)\,.$$
Thus we have
\begin{multline*}
\frac{2}{m_i}\,\sin\frac{\theta_i + \theta_{i-1}}{2} = \frac{4\sin \frac{\theta_i + \theta_{i-1}}{2}}{k(\cos\theta_{i-1} - \cos\theta_i)} = \frac{4\sin \frac{\theta_i + \theta_{i-1}}{2}}{k \sin \frac{\theta_i + \theta_{i-1}}{2} \sin \frac{\theta_i - \theta_{i-1}}{2}} = \\ = \frac{4}{k \sin \frac{\theta_i - \theta_{i-1}}{2}} \geq \frac{8}{k(\theta_i - \theta_{i-1})}\,.
\end{multline*}
Equation (2.20) in \cite{Zhou_Arrangements_of_points_on_the_sphere} reads
$$\frac{C_3}{\sqrt k} \leq \theta_i - \theta_{i-1} \leq \frac{C_4}{\sqrt k}\,,$$
where $C_3 = \beta - \delta_1-\delta_2$ and $C_4 = \pi\sqrt r + \delta_1 + \delta_2$. The constants $\beta,\delta_1,\delta_2$ are defined in equations (2.5), (2.12), (2.13), respectively, and $r = 2$. Combined with our estimates so far, this means that
$$D((\theta,\phi),\partial D_{ij}) \geq \min\left(\frac{C_3}{2}\,,\frac{8}{C_4}\right)\cdot\frac 1 {\sqrt k}\,.$$

This holds for all the regions in the bands between $\theta_{i-1}$ and $\theta_i$ for $1\leq i \leq n$. What remains are the polar bands, given by the rectangles
$$[0,\theta_0] \times \left[\frac{\pi(j-1)}{3},\frac{\pi j}{3}\right]\quad\text{ and }\quad[\pi-\theta_0,\pi] \times \left[\frac{\pi(j-1)}{3},\frac{\pi j}{3}\right]\,,\quad j=1,\dots, 6\,.$$
Letting $(\theta,\phi)$ be the center of one of these regions in spherical coordinates, the above discussion shows that its spherical distance to the boundary of the region is at least
$$\min\left(\frac{\theta_0}{2}\,,\frac 1 2\sin\frac{\theta_0}{2}\right) = \frac 1 2 \sin \frac{\theta_0}{2} = \frac 1 2 \sqrt{\frac 1 2 (1-\cos \theta_0)} = \sqrt{\frac 3 2}\cdot \frac 1 {\sqrt k}\,.$$
Let
$$C_5 = \min\left(\frac{C_3}{2}\,,\frac{8}{C_4}\,,\sqrt{\frac 3 2}\right)\,.$$
It follows that every region in the subdivision contains a spherical cap with spherical radius
$$\frac{C_5}{\sqrt k}\,,$$
and therefore a spherical cap with \emph{Euclidean} radius at least
$$\frac{2C_5}{\pi}\cdot\frac{1}{\sqrt k}\,.$$
We have
$$\frac{2C_5}\pi \approx 0.77970\dots$$
\end{proof}

\begin{proof}[Proof of Lemma \ref{lemma:complexity_finding_region_contain_given_point}]
The subdivision algorithm of \cite{Zhou_Arrangements_of_points_on_the_sphere} produces approximate latitudes $\theta_i'$, \emph{which form an arithmetic progression} and exact latitudes $\theta_i$ for the parallels bounding the bands which are then subdivided into the regions. There are estimates \cite[Equation (2.20)]{Zhou_Arrangements_of_points_on_the_sphere}:
$$|\theta_i - \theta_i'| \leq \frac{C_6}{\sqrt k}\,,\qquad \theta_i - \theta_{i-1} \geq \frac{C_3}{\sqrt k}\,.$$
Given the latitude $\theta$ of the given point in $S^2$, we can find the interval $[\theta_{i-1}',\theta_i']$ to which it belongs. Since the angles $\theta_i'$ are an arithmetic progression, that is $\theta_i' = \theta_0' + i\Delta\theta$, finding the appropriate $i$ amounts to computing the integer part of a number, which has complexity $O(1)$. The actual interval we are looking for is $[\theta_{j-1},\theta_j]$. The above estimates tell us that the number $j$ of the corresponding interval differs from the value $i$ found in the previous calculation by at most a constant. We can then enumerate all the candidate intervals to find which one contains $\theta$. This also has complexity $O(1)$, since the number of intervals we have to enumerate is bounded above by a constant independent of $\theta$. Finally, since the azimuths of the regions form another arithmetic progression, namely $\frac{2\pi j}{m_i}$, finding the corresponding interval of azimuthal angles which contains the azimuth $\phi$ of the given point amounts to computing the integer part of a number is therefore it has complexity $O(1)$.
\end{proof}

\begin{proof}[Proof of Lemma \ref{lemma:complexity_computing_subdivision}]
As mentioned in the proof of Lemma \ref{lemma:lower_bound_radius_spher_cap_region}, the regions of the subdivision are rectangles with respect to the spherical coordinates. The latitudes of the parallels bounding the bands are obtained in two stages. The total number of bands is approximately $\sqrt{k/2}$. In the first stage, one finds approximate angles, which are equally spaced, for which the complexity is $O(\sqrt k)$. Then one corrects the approximate angles to exact ones, using Lemma 2.11 of \cite{Zhou_Arrangements_of_points_on_the_sphere}. The lemma uses an inductive procedure which runs on a list of length $O(\sqrt k)$, and at every step corrects the two outermost entries in the list on each side. The total number of operations therefore is linear in the size of the list and thus has complexity $O(\sqrt k)$. Finally each band is subdivided into the required number of regions. Computing the azimuthal angles of the regions therefore has complexity proportional to the total number of the regions, that is $O(k)$. The total complexity therefore is $O(k)$.
\end{proof}

\begin{proof}[Proof of Proposition \ref{prop:pties_PL_approx_sphere}]Keep the notations of Section \ref{sss:estimate}. For a linear simplex $\sigma$ let $\wt\sigma = r(\sigma)$ be the corresponding curvilinear simplex on $S^2$.

For (i), let $z \in S^2$ and assume that $\sigma = [z_0,z_1,z_2]$ is a linear simplex with $z \in \wt\sigma$. Without loss of generality we assume that $f(z_0) \leq f(z_1) \leq f(z_2)$. Since $F|_{\wt \sigma} = \wt F_\sigma \circ r_\sigma^{-1}$ and $\wt F_\sigma \fc \sigma \to \R$ is an affine function, it follows that $F(z) = \wt F_\sigma(r_\sigma^{-1}(z)) \in [f(z_0),f(z_2)]$, thus
$$F(z) - f(z) \leq f(z_2) - f(z) \leq \|f\|_{\Lip}\cdot d(z,z_2) \leq \|f\|_{\Lip}\cdot\diam\wt\sigma$$
and
$$F(z) - f(z) \geq f(z_0) - f(z) \geq -\|f\|_{\Lip}\cdot d(z,z_0) \geq -\|f\|_{\Lip}\cdot\diam\wt\sigma\,,$$
and we are done.

For (ii), let $\Sigma$ be the union of the linear simplexes $\sigma$ of the triangulation, let $r_\Sigma \fc \Sigma \to S^2$ be the restriction of the radial projection map, and let $\wt F \fc \Sigma \to \R$ be defined by $\wt F|_{\sigma} = \wt F_\sigma$ for all $\sigma$. Note that $\Sigma$ is a polyhedron, $r_\Sigma$ is a homeomorphism, $\wt F$ is a well-defined PL function on $\Sigma$, and $F = \wt F \circ r_\Sigma^{-1}$.

For $z,z' \in \Sigma$, define $d_\Sigma(z,z')$ to be the length of a shortest path on $\Sigma$ between $z,z'$ of the form $r_\Sigma^{-1}\circ\gamma$ where $\gamma$ is a round geodesic on $S^2$ between $r_\Sigma(z)$ and $r_\Sigma(z')$. It is not hard to see that $d_\Sigma$ is a metric on $\Sigma$. Due to our assumptions on the triangulation it follows that $d = d_\Sigma$ on every linear simplex $\sigma$, where $d$ is the Euclidean metric on $\R^3$.

Note that $r_\Sigma^{-1} \fc (S^2,D) \to (\Sigma,d_\Sigma)$ is a contraction, that is $1$-Lipschitz, where $D$ is the round metric on $S^2$. Since $D \leq \frac \pi 2 d$, it follows that $r_\Sigma^{-1} \fc (S^2,d) \to (\Sigma,d_\Sigma)$ is $\frac{\pi}{2}$-Lipschitz, and therefore
$$\|F\|_{\Lip,d} \leq \frac{\pi}{2}\cdot\|\wt F\|_{\Lip,d_\Sigma}\,,$$
where we wrote the metrics explicitly. It therefore suffices to show that
$$\|\wt F\|_{\Lip,d_\Sigma} \leq \frac{1}{\sin(\theta_0/2)}\cdot \|f\|_{\Lip,d}\,,$$
where $\theta_0$ is the angle appearing in the formulation.

We claim that
$$\|\wt F\|_{\Lip,d_\Sigma} = \max_\sigma\|\wt F_\sigma\|_{\Lip,d_\Sigma} = \max_\sigma\|\wt F_\sigma\|_{\Lip,d}\,.$$
Clearly the left-hand side is $\geq$ the right-hand side. To see the opposite inequality, let $z,z'\in\Sigma$ be two distinct points and let $\delta \fc [0,1] \to \Sigma$ be a shortest path on $\Sigma$ with $\delta(0) = z$, $\delta(1) = z'$. Let $\sigma_1,\dots,\sigma_k$ be the distinct linear simplexes which intersect the image of $\delta$ along the interior or along the interior of an edge. Renaming them if necessary, we can assume that there are points $t_0 = 0 < t_1 < \dots < t_k = 1$ such that $\delta([t_j,t_{j+1}]) \subset \sigma_{j+1}$ for $j = 0,\dots,k-1$. Note that the same simplex cannot intersect $\delta$ in two different segments, because $r_\Sigma \circ \delta$ is a shortest round geodesic on $S^2$ between $r_\Sigma(z),r_\Sigma(z')$ and the images of the linear simplexes by $r_\Sigma$ are geodesically convex subsets of $S^2$, therefore they can only intersect a shortest geodesic along a connected segment.

Let us denote $z_j = \delta(t_j)$ for $j = 0,\dots,k$. Note that since $\delta|_{[t_j,t_{j+1}]}$ is a straight segment, it is the shortest path between $z_j,z_{j+1}$, and therefore $d_\Sigma(z_j,z_{j+1})$ is the length of $\delta|_{[t_j,t_{j+1}]}$. Since $d_\Sigma(z,z')$ is the length of $\delta$, which in turn is the sum of the lengths of the $d_\Sigma(z_j,z_{j+1})$, which as we just saw is the sum of $d_\Sigma(z_j,z_{j+1})$, we obtain
$$d_\Sigma(z,z') = \sum_{j=0}^{k-1}d_\Sigma(z_j,z_{j+1})\,.$$
It follows that
\begin{multline*}
|\wt F(z) - \wt F(z')| \leq \sum_{j=0}^{k-1}|\wt F(z_j) - \wt F(z_{j+1})|
 \leq \sum_{j=0}^{k-1} \|\wt F_{\sigma_{j+1}}\|_{\Lip,d_{\Sigma}}\cdot d_\Sigma(z_{j},z_{j+1})\\
\leq \max_\sigma\|\wt F_\sigma\|_{\Lip,d_\Sigma}\sum_{j=0}^{k-1} d_\Sigma(z_{j},z_{j+1}) = \max_\sigma\|\wt F_\sigma\|_{\Lip,d_\Sigma}\cdot d_\Sigma(z,z')\,,
\end{multline*}
as claimed.

It remains to show that $\|\wt F_\sigma\|_{\Lip,d} \leq \frac{\|f\|_{\Lip,d}}{\sin(\theta_0/2)}$. Let $\sigma = [z_0,z_1,z_2]$. For the purpose of estimating the Lipschitz constant, we can translate and rotate $\sigma$ so that it is contained in $\R^2 \times \{0\} \subset \R^3$. We can further assume that relative to the coordinates on $\R^2$, we have $z_0 = 0$, $z_1 = (a,0)$, $z_2 = (b,c)$, with $a,c > 0$, and, subtracting a constant from $f$, we can assume that $f(z_0) = 0$. Denote $q_1=f(z_1)$, $q_2 = f(z_2)$. Let $G \fc \R^2 \to \R$ be the the linear function such that $G(z_1) = q_1$, $G(z_2) = q_2$. Since $G|_\sigma = \wt F_\sigma$, the Lipschitz constants of $\wt F_\sigma$ and $G$ are the same. Since $G \fc \R^2 \to \R$ is linear, its Lipschitz constant is
$$\sqrt{G^2(1,0)+G^2(0,1)}\,.$$
We have
$$(1,0) = \frac {z_1} a\,,\quad \text{ thus }\quad G(1,0) = \frac {q_1} a = \frac{q_1}{\|z_1\|}\,. $$
Next,
$$(0,1) = -\frac{b}{ac}\cdot z_1 + \frac 1 c \cdot z_2\,,\quad \text{ thus } \quad G(0,1) = -\frac{bq_1}{ac} + \frac{q_2}{c} = -\frac{q_1}{\|z_1\|} \cot\theta + \frac{q_2}{\|z_2\|\sin\theta}\,,$$
where $\theta$ is the angle between $z_1$ and $z_2$. We have
$$G^2(1,0) + G^2(0,1) = \frac{q_1^2}{\|z_1\|^2} + \frac{q_1^2}{\|z_1\|^2}\cot^2\theta - \frac{2q_1q_2\cot\theta}{\|z_1\|\|z_2\|\sin\theta}+\frac{q_2^2}{\|z_2\|^2\sin^2\theta}\,.$$
Let $L = \|f\|_{\Lip,d}$. Then $|q_j| = |f(z_j) - f(z_0)| \leq L\|z_j - z_0\| = L \|z_j\|$ for $j = 1,2$. Moreover, since by assumption $\theta \in [\theta_0,\pi - \theta_0]$, we have $\sin\theta \geq \sin\theta_0$ and $|\cot\theta| \leq \cot\theta_0$. Plugging these estimates into the last expression, we obtain
$$\|\wt F_\sigma\|_{\Lip,d}^2 = G^2(1,0) + G^2(0,1) \leq L^2\left(1 + \cot^2\theta_0 + 2\frac{\cot\theta_0}{\sin\theta_0} + \frac{1}{\sin^2\theta_0}\right) = \frac{L^2}{\sin^2(\theta_0/2)}\,,$$
and the proof of the proposition is thereby complete.
\end{proof}

\begin{proof}[Proof of Lemma \ref{lemma:estimates_icosahedral_triangulation}] For (i), let $\wh\sigma$ be one of the $20N^2$ congruent triangles. The corresponding curvilinear simplex is $\wt\sigma = r(\wh\sigma)$ where $r \fc \R^3 - \{0\} \to S^2$ is the radial projection. Since $r$ is $C^1$, its Lipschitz constant can be estimated in terms of the operator norm of its differential, namely if we restrict $r$ to the complement of the Euclidean ball $B(\rho) \subset \R^3$ of radius $\rho > 0$, then
$$\|r\|_{\Lip(\R^3 - B(\rho),d)} \leq \sup_{x \notin B(\rho)}\|d_xr\|\,,$$
where
$$\|d_xr\| = \max_{\|X\| = 1}\|d_xr(X)\|\,.$$
An explicit calculation shows that $\|d_xr\| = \frac{1}{\|x\|}$, and therefore
$$\|r\|_{\Lip(\R^3 - B(\rho),d)} \leq \frac{1}{\rho}\,.$$
It follows that
$$\diam \wt\sigma \leq \|r\|_{\Lip(\wh\sigma,d)}\diam \wh\sigma \leq \frac{1}{c_{\wh\sigma}}\diam\wh\sigma\,,$$
where $c_{\wh\sigma}$ is the distance of the origin $0 \in \R^3$ to the closest point on $\wh\sigma$. We have $c_{\wh\sigma} \geq R_i$ where $R_i$ is the insphere radius of the icosahedron.

Next, since $\wh\sigma$ is one of the $N^2$ congruent triangles into which we subdivided a face of the icosahedron, its diameter is $\frac{1}{N}$ times the diameter of the face, which equals its edge length $a$, since it is an equilateral triangle. In total
$$\diam \wt\sigma \leq \frac{a}{R_i}\frac 1 N\,.$$
The ratio of the edge length to the insphere radius for the icosahedron is $\sqrt 3\big(3 - \sqrt 5\big)$, thus
$$\diam\wt\sigma \leq \frac{\sqrt 3\big(3-\sqrt 5\big)}{N}\,.$$

For (ii) we will present a general argument for estimating the angles of a linear simplex whose vertices are obtained by radially projecting three points lying inside the $3$-ball bounded by $S^2$. Let $z \in (0,1)$ be fixed and consider the disk $D = \R^2 \times \{z\} \cap \ol B{}^3(1)$, where $\ol B{}^3(1)$ is the closed Euclidean ball of radius $1$. Consider a triangle $\Delta \subset D$ with vertices $x_i$ and let $y_i = r(x_i)$. In our case $\Delta$ is equilateral so all of its interior angles are $\frac \pi 3$, and we would like to estimate the angles of the Euclidean triangle $\wt\Delta$ with vertices $y_1,y_2,y_3$, more specifically we wish to obtain a lower estimate on these angles. The strategy is to obtain \emph{upper bounds} on these angles, since if every angle is at most $\theta$, then every angle is at least $\pi -2\theta$.

First, note that each interior spherical angle of the spherical triangle $r(\Delta)$ is at least the angle of $\wt\Delta $ at the corresponding vertex. Therefore to obtain upper bounds on the angles of $\wt\Delta$ we can bound the spherical angles of $r(\Delta)$ from above. This can be done by considering the differential of $r$. Let $x$ be one of the vertices of $\Delta$ and let $X,X' \in T_xD = \R^2 \times \{0\}$ be two unit vectors tangent to $D$ pointing from $x$ in the directions of the other two vertices. Then
$$d_xr(X) = \frac{X}{\|x\|} - \frac{\langle x,X\rangle}{\|x\|^3}\;x$$
and similarly for $d_xr(X')$. Letting $\beta$ be angle between $X,X'$, $\alpha,\alpha'$ the angles between $x$ and $X,X'$, respectively, and $\gamma$ the angle between $d_xr(X),d_xr(X')$, which is the spherical angle of $r(\Delta)$ at $r(x)$, an elementary calculation shows that
$$\cos \gamma = \frac{\cos \beta - \cos\alpha\cos\alpha'}{\sin\alpha\sin\alpha'}\,.$$
Since we wish to obtain an \emph{upper} bound on $\gamma$, we need to establish a \emph{lower} bound on $\cos \gamma$. Using the fact that $x$ is constrained to the disk $D$, we can prove that $|\cos\alpha|,|\cos\alpha'| \leq \sqrt{1-z^2}$. Since $\beta = \frac{\pi}{3}$ by assumption, we have $\cos \beta = \frac 1 2$. Moreover writing $\cos \alpha = s$, $\cos\alpha' = s'$, we come to the minimization problem of the function
$$\frac{\tfrac 1 2 - ss'}{\sqrt{(1-s^2)(1-s'^2)}}$$
over the square $(s,s') \in [-\sqrt{1-z^2},\sqrt{1-z^2}]^2$. Since the angles $\alpha,\alpha'$ are not independent, because the angle between $X,X'$ is set to $\frac{\pi}{3}$, the actual lower bound on $\cos \gamma$ is going to be larger than the absolute minimum of this function, but the latter is much more easily computed, and in fact it is attained at the point $s=s' = \sqrt{1-z^2}$, therefore it equals $1-\frac{1}{2z^2}$, so in total we obtain
$$\cos\gamma \geq 1 - \frac{1}{2z^2}\,.$$
For the icosahedral triangulation all the triangles lie in planes which are at a distance from $0 \in \R^3$ at least the insphere radius of the icosahedron, which is $\frac{3+\sqrt 5}{\sqrt{30 + 6\sqrt 5}}$, therefore $z$ can be set to this number, which implies that the spherical angles of the projection of one of the triangles of the triangulation are all at most $\arccos\big(1 - \frac 1 {2z^2}\big)$ radians. By the above, this also is an upper bound on the angles of the corresponding Euclidean triangle, which in turn implies that all the angles are \emph{at least}
$$\theta_0 = \pi - 2 \arccos \left(1-\frac 1{2z^2}\right) = \pi-2\arccos\left(\frac 1 2 \left(6\sqrt 5 - 13\right)\right)\,.$$
\end{proof}

\subsection{Approximation results}\label{ss:approx_results}

We start by proving Theorems \ref{thm:metric_conv_qss} and \ref{thm:weak_conv_qss}. For this we will need the results of Section \ref{sss:simples_qss_mfds} as well as the following elementary fact: if $\tau$ is any topological measure on $X$, then any two closed subsets $C,C'\subset X$ with $\tau(C) = \tau(C') = 1$ must intersect. Indeed, otherwise by the additivity and monotonicity of $\tau$ we would have
$$2 = \tau(C) + \tau(C') = \tau(C \cup C') \leq \tau(X) = 1\,,$$
which is a contradiction.

\begin{proof}[Proof of Theorem \ref{thm:metric_conv_qss}]
Let us first prove the theorem for a Morse function $f \in C(X)$. Let $m = m_\mu(f)$. Let $\epsilon > W_\infty(\mu,\nu)$. Let $M$ be the closed $\epsilon$-neighborhood of $m$. The complement $X - M$ is a countable disjoint union of open connected sets $D_i$. Let $X - m = C_1 \cup \dots \cup C_k$ be the decomposition into connected components. Since $X - M \subset X - m$, every $D_i$ is contained in one of the $C_j$.

For a Borel subset $A \subset X$ let us denote
$$A^{\leq \epsilon} = \{x \in X\,|\,d(x,y) \leq \epsilon\text{ for some }y\in A\}\,.$$
Fix one of the components $D_i$, and assume $D_i \subset C_j$. We see that $D_i^{\leq \epsilon}$ is connected and disjoint from $m$. Since $D_i \subset C_j$ and $C_j$ is itself connected and disjoint from $m$, it follows that $D_i^{\leq \epsilon} \subset C_j$.

Proposition 5 of \cite{Givens_Shortt_A_class_of_Wasserstein_metrics_prob_distr} exhibits another formula for $W_\infty$:
$$W_\infty(\mu,\nu) = \inf\{\epsilon > 0\,|\, \mu(A^{\leq \epsilon}) \geq \nu(A)\text{ for all Borel }A \subset X\}\,.$$
It follows that
$$\nu(D_i)\leq\mu(D_i^{\leq \epsilon}) \leq \mu(C_j)\leq \frac 1 2\,,$$
implying that $\tau_\nu(D_i) = 0$. From the countable additivity of $\tau_\nu$ \cite{Grubb_LaBerge_Additivity_qms} it follows that $\tau_\nu(M) = 1$. Since $\tau_\nu(m_\nu(f)) = 1$, it follows that $M \cap m_\nu(f) \neq \varnothing$. Since $\max_{x\in M}d(x,m) \leq \epsilon$ and since $f$ is Lipschitz, we obtain
$$|\zeta_\mu(f) - \zeta_\nu(f)| = |f(m) - f(m_\nu(f))| \leq \|f\|_{\Lip}\cdot \epsilon\,.$$
The assertion now follows from the definition of $W_\infty(\mu,\nu)$, \eqref{eq:W_infty_metric}.

If $f\fc X \to \R$ is an arbitrary Lipschitz function, for every $\delta>0$ there is a Morse function $g \fc X \to \R$ with $\|g\|_{\Lip} \leq \|f\|_{\Lip}+\delta$ and $\|f-g\|_{C^0} \leq \delta$. It then follows that
$$|\zeta_\mu(f) - \zeta_\nu(f)| \leq |\zeta_\mu(f) - \zeta_\mu(g)| + |\zeta_\mu(g) - \zeta_\nu(g)| + |\zeta_\nu(g) - \zeta_\nu(f)|\,.$$
Since quasi-states are $1$-Lipschitz with respect to the $C^0$-norm, the first and third terms in the right-hand side are at most $\delta$, and for the middle term we have just established an upper bound of $\|g\|_{\Lip}\cdot W_\infty(\mu,\nu)$, thus
$$|\zeta_\mu(f) - \zeta_\nu(f)| \leq (2+W_\infty(\mu,\nu))\delta + \|f\|_{\Lip}\cdot W_\infty(\mu,\nu)\,,$$
and the result follows upon taking $\delta \to 0$.
\end{proof}

\begin{proof}[Proof of Theorem \ref{thm:weak_conv_qss}]The space $\cP_0(X)$ being metrizable, it suffices to show that if $\mu_n \in \cP_0(X)$ is a sequence converging to $\mu \in \cP_0(X)$, then $\zeta_{\mu_n}(f) \to \zeta_\mu(f)$ for every $f \in C(X)$.

Let us first prove the theorem when $f$ is a Morse function. Let $m = m_\mu(f)$. Let $\{C_i\}_{i=1}^k$ be the connected components of $X - m$. Recall that $\mu(C_i) \leq \frac 1 2$ for all $i$. Let $\Gamma$ be the Reeb graph of $f$ and let $\pi \fc X \to \Gamma$ be the quotient map. Let $\wh f \fc \Gamma \to \R$ be the continuous function such that $\wh f \circ \pi = f$. The function $\wh f$ induces a metric $d_f$ on $\Gamma$, where $d_f(a,b)$ is the infimum of total variations of $\wh f$ over all paths connecting $a,b$.

Fix $\epsilon > 0$. Let $M$ be the preimage by $\pi$ of the open $\epsilon$-neighborhood of $m$ in $\Gamma$ with respect to $d_f$, and let $B_i = C_i \cap (X - M)$. Then $\mu(B_i) < \frac 1 2$. Indeed, otherwise we would have $\mu\big(m \cup \bigcup_{j\neq i}C_j\big) = \frac 1 2$ and $\mu(B_i) = \frac 1 2$, and since $B_i$ and $m \cup \bigcup_{j\neq i}C_j$ are closed and disjoint, this would contradict the assumption $\frac 1 2 \notin \Spec \mu$.

Note that $B_i$ is connected. Since it is also closed, it follows from the portmanteau theorem \cite{Bhattacharya_Waymire_Weak_conv_prob_meas_metric_spaces} that $\limsup_{n\to\infty}\mu_n(B_i) \leq \mu(B_i) < \frac 1 2$, and therefore that there is $n_0 \in \N$ such that for all $n \geq n_0$ and $i = 1,\dots,k$ we have $\mu_n(B_i) < \frac 1 2$. Since the complement of $B_i$ is connected as well, we see that $B_i$ is solid, and therefore from \eqref{eq:defin_tm_corresp_to_meas} it follows that $\tau_{\mu_n}(B_i) = 0$, and therefore from additivity that $\tau_{\mu_n}(M) = 1$. From monotonicity it follows that $\tau_{\mu_n}(\ol M) = 1$. Since $\tau_{\mu_n}(m_{\mu_n}(f)) = 1$, from the discussion at the beginning of the subsection it follows that $m_{\mu_n(f)} \cap \ol M \neq \varnothing$. Since $\max_{x\in \ol M}|f(x) - f(m)| \leq \epsilon$, we see that
$$|\zeta_{\mu_n}(f) - \zeta_\mu(f)| = |f(m_{\mu_n}(f)) - f(m)| \leq \epsilon\,,$$
which proves that $\zeta_{\mu_n}(f) \to \zeta_\mu(f)$.

Let now $f \in C(X)$ be arbitrary. Then for $\epsilon > 0$ there is a Morse function $g$ with $\|f-g\|_{C^0} < \epsilon$. We have
$$|\zeta_{\mu_n}(f) - \zeta_\mu(f)| \leq |\zeta_{\mu_n}(f) - \zeta_{\mu_n}(g)| + |\zeta_{\mu_n}(g) - \zeta_\mu(g)| + |\zeta_\mu(g) - \zeta_\mu(f)|\,,$$
and owing to the Lipschitz property of quasi-states we see that the first and the last terms in the right-hand side are $< \epsilon$. Taking $\limsup$ with respect to $n$, we obtain
$$\limsup_{n\to\infty}|\zeta_{\mu_n}(f) - \zeta_\mu(f)| \leq 2\epsilon + \lim_{n\to\infty}|\zeta_{\mu_n}(g) - \zeta_\mu(g)| = 2\epsilon\,,$$
and since $\epsilon$ is arbitrary, we arrive at the desired conclusion $\lim_{n\to\infty}|\zeta_{\mu_n}(f) - \zeta_\mu(f)| = 0$.
\end{proof}

We also prove Proposition \ref{prop:W_1_induces_weak_top_qss}, which states that the $W_1$ distance induces the weak topology on $\cQ(X)$.
\begin{proof}[Proof of Proposition \ref{prop:W_1_induces_weak_top_qss}]In one direction, assume that a net of quasi-states $\zeta_i \in \cQ(X)$ converges to $\zeta \in \cQ(X)$ with respect to $W_1$, that is for every Lipschitz $f$ we have $|\zeta_i(f) - \zeta(f)| \leq W_1(\zeta_i,\zeta)\cdot\|f\|_{\Lip}$, therefore $\zeta_i(f) \to \zeta(f)$. By the Stone--Weierstra\ss\ theorem, $\Lip(X,d)$ is dense in $C(X)$, therefore for every $f \in C(X)$ and $\delta > 0$ we can find $g \in \Lip(X,d)$ with $\|f - g\|_{C^0} < \delta$, therefore
$$|\zeta_i(f) - \zeta(f)| \leq |\zeta_i(f) - \zeta_i(g)| + |\zeta_i(g) - \zeta(g)| + |\zeta(g) - \zeta(f)|\,.$$
Quasi-states being $1$-Lipschitz, the first and the third terms on the right are $\leq \delta$, while the middle term goes to zero as $i \to \infty$, because $g$ is Lipschitz, therefore
$$\limsup_i|\zeta_i(f) - \zeta(f)| \leq 2\delta\,,$$
and taking $\delta \to 0$, we see that $\zeta_i(f)\to\zeta(f)$. Thus $\zeta_i \to \zeta$ weakly.

For the other direction, note that if we fix $x_0 \in X$, then the subspace
$$\cL = \{f \in \Lip(X,d)\,|\,f(x_0) = 0\,,\|f\|_{\Lip} \leq 1\} \subset C(X)$$
is compact due to the Arzel\`a--Ascoli theorem. It follows that in \eqref{eq:W_1_dual_defin_via_Lip_fcns} the supremum is attained, that is for every $\zeta,\zeta' \in \cQ(X)$ there is $f \in \cL$ such that
$$W_1(\zeta,\zeta') = \zeta(f) - \zeta'(f)\,.$$
Assume now that $\zeta_i \to \zeta$ weakly. It is enough to prove that $\limsup_iW_1(\zeta_i,\zeta) = 0$. Let us pass to a subnet $\zeta_{i_k}$ such that
$$\limsup_iW_1(\zeta_i,\zeta) = \lim_kW_1(\zeta_{i_k},\zeta)\,,$$
and let us prove that the latter limit is zero. To simplify notation, we will write $\zeta_i$ for this subnet as well.

For every $i$ fix $f_i \in \cL$ with
$$W_1(\zeta_i,\zeta) = \zeta_i(f_i) - \zeta(f_i)\,.$$
Since $\cL$ is compact, there is a subnet of $(f_i)_i$ converging to $f \in \cL$. By abuse of notation we denote this subnet again by $(f_i)_i$. Then we have
\begin{align*}
W_1(\zeta_i,\zeta) &= \zeta_i(f_i) - \zeta(f_i)\\
&= [\zeta_i(f_i) - \zeta_i(f)] + [\zeta(f) - \zeta(f_i)] + [\zeta_i(f) - \zeta(f)]\\
&\leq 2\|f - f_i\|_{C_0} + |\zeta_i(f) - \zeta(f)|\,,
\end{align*}
where we again used the Lipschitz property of quasi-states. Now $\|f_i-f\|_{C^0} \to 0$ by the definition of $f$, and $|\zeta_i(f) - \zeta(f)| \to 0$ by the weak convergence of $\zeta_i$ to $\zeta$. Thus $W_1(\zeta_i,\zeta) \to 0$ and the proposition is proved.
\end{proof}

\subsection{Non-approximation results}\label{ss:pfs_non_approx_results}

For the proof of Theorem \ref{thm:general_non_approximation} we need the following elementary lemma.
\begin{lemma}\label{lemma:meas_thry_empty_triple_intersections}
Let $(\Omega,\cB,\mu)$ be a probability space. Assume there are $U_1,\dots,U_4 \in \cB$ such that the triple intersections are all empty. Then there is $j =1,\dots,4$ such that $\mu(U_j) \leq \frac 1 2$.
\end{lemma}
\begin{proof}
Assume the contrary: $\mu(U_j) > \frac 1 2$ for all $j$. Since the triple intersections are empty, we have
$$1 \geq \mu(U_1 \cup U_2 \cup U_3 \cup U_4) = \sum_{i=1}^4\mu(U_i) - \sum_{i < j}\mu(U_i\cap U_j) > 2 - \sum_{i < j}\mu(U_i\cap U_j)\,,$$
which implies
$$\sum_{i < j}\mu(U_i\cap U_j) > 1\,,$$
which is a contradiction since the sets $U_i \cap U_j$ are pairwise disjoint for $i < j$.
\end{proof}

\begin{proof}[Proof of Theorem \ref{thm:general_non_approximation}]For a subset $A \subset X$ we let the open $\epsilon$-neighborhood of $A$ be denoted by $A^{<\epsilon}$.

There is $\epsilon > 0$ such every $Y_j^{<\epsilon}$ is a tubular neighborhood of $Y_j$, and such that the triple intersections of $Y_1^{<\epsilon},\dots,Y_4^{<\epsilon}$ are still empty. Fix $\mu \in \cP_0(X)$. It follows from Lemma \ref{lemma:meas_thry_empty_triple_intersections} that there is $j = 1,\dots,4$ such that $\mu(Y_j^{<\epsilon}) \leq \frac 1 2$. Since $Y_j$ is assumed to have codimension at least $2$, the tubular neighborhood $Y_j^{<\epsilon}$ has connected complement and is therefore solid. From \eqref{eq:defin_tm_corresp_to_meas} it follows that $\tau_\mu(Y_j^{<\epsilon}) = 0$. Thus $\tau_\mu(Y_j^{<s}) = 0$ and $\tau_\mu(X - Y_j^{<s}) = 1$ for all $s \in [0,\epsilon]$.

Let us define $\rho \fc X \to \R$ to be the distance from $Y_j$. Then $Y_j^{<s} = \{\rho < s\}$ and thus $X - Y_j^{<s} = \{\rho\geq s\}$. We then have
$$\zeta_\mu(\rho) = \int_0^{\max \rho} \tau_\mu(\{\rho \geq s\})\,ds \geq \epsilon\,,$$
since the integrand is $1$ for $s \in [0,\epsilon]$. On the other hand, since $\tau(Y_j) = 1$ and $\rho|_{Y_j} \equiv 0$, we have $\zeta(\rho) = 0$. Therefore, since $\rho$ is $1$-Lipschitz, we have
$$W_1(\zeta,\zeta_\mu) \geq \zeta_\mu(\rho) - \zeta(\rho) \geq \epsilon$$
as claimed.
\end{proof}

\begin{proof}[Proof of Theorem \ref{thm:non_approx_sympl_mfds}] Let $U \subset X$ be a Weinstein neighborhood of $L$ (see \cite[Theorem 3.33]{McDuff_Salamon_Intro_sympl_topology}), that is there is an open subset $V \subset T^*L$ containing $L$ and a symplectomorphism $\psi \fc U \to V$ where $V \subset T^*L$, and $\psi|_L = \id_L$. Let $f_0 \fc L \to \R$ be the zero function. For generically chosen Morse functions $f_1,f_2,f_3$, the functions $f_i-f_j$, $0 \leq i < j\leq 3$ all have pairwise disjoint critical point sets. In particular, the Lagrangian submanifolds $\Gamma_{df_i}$, $i = 0,\dots,3$ have empty triple intersections, where for a $1$-form $\alpha$ on $L$ we let $\Gamma_\alpha$ be the image of $\alpha$ viewed as a section $\alpha \fc L \to T^*L$. Indeed, for two $1$-forms $\alpha,\beta$ on $L$ we have $\Gamma_\alpha \cap \Gamma_\beta = \{(q,\xi)\in T^*L\,|\,\alpha_q = \beta_q = \xi\}$, which is in bijection with the zero set of $\alpha - \beta$.

Scaling if necessary, we can assume that $\Gamma_{df_i} \subset V$ for all $i$, and so we can define Lagrangian submanifolds $L_i:=\psi^{-1}(\Gamma_{df_i}) \subset X$. It is a standard fact that $L_i$ are Hamiltonian isotopic to $L$. Therefore we have found a quadruple of Lagrangian submanifolds in $X$, all Hamiltonian isotopic to $L$, with empty triple intersections. Since $\zeta$ and $\tau$ are invariant under the action of the Hamiltonian group of $X$, it follows that $\tau(L_i) = \tau(L) = 1$ for $i=1,2,3$. Since Lagrangians have codimension $\frac 1 2 \dim M \geq 2$, we see that all the conditions of Theorem \ref{thm:general_non_approximation} are satisfied, and the conclusion follows.
\end{proof}

\begin{proof}[Proof of Corollary \ref{coroll:non_approx_CPn_S2_x_S2}]If $X = \C P^n$, it is known that the topological measure $\tau$ associated to $\zeta$ satisfies $\tau(T) = 1$ where $T \subset \C P^n$ is the Clifford torus \cite{Entov_Polterovich_rigid_subsets_sympl_mfds}. If $X = S^2 \times S^2$, then for the product $L$ of equators we have $\tau(L) = 1$, \emph{ibid}. The result now follows from Theorem \ref{thm:non_approx_sympl_mfds}, using the fact that $\zeta$ is invariant under the Hamiltonian group.
\end{proof}

\bibliography{biblio}
\bibliographystyle{plain}

\end{document}